\documentclass[11pt]{article} 

\usepackage[utf8]{inputenc} 

\usepackage{geometry} 
\usepackage{graphicx} 

\usepackage{booktabs} 
\usepackage{array} 
\usepackage{paralist} 
\usepackage{verbatim} 
\usepackage{subfig} 

\usepackage{fancyhdr} 
\usepackage{sectsty}
\allsectionsfont{\sffamily\mdseries\upshape} 

\usepackage[nottoc,notlof,notlot]{tocbibind} 
\usepackage[titles,subfigure]{tocloft} 

\usepackage{amsmath,amsthm, amssymb} 
\usepackage{enumerate}
\usepackage{booktabs}

\numberwithin{equation}{section} 

\theoremstyle{plain}\newtheorem{thm}{Theorem}[section]
\theoremstyle{plain}\newtheorem{defn}[thm]{Definition}

\theoremstyle{plain}\newtheorem{lem}[thm]{Lemma}
\theoremstyle{plain}
\theoremstyle{plain}\newtheorem{prop}[thm]{Proposition}
\theoremstyle{remark}\newtheorem{rmk}[thm]{Remark}

\numberwithin{thm}{section}

\title{H\"older continuity of bounded, weak solutions of a variational system in the critical case}
\author{Nirav Shah}
\date{} 

\begin{document}
\maketitle

\begin{abstract}
\noindent Let $\Omega\subset\mathbb{R}^{2}$ be a bounded, Lipschitz domain. We consider bounded, weak solutions ($u\in W^{1, 2}\cap L^{\infty}(\Omega;\mathbb{R}^N)$) of the vector-valued, Euler-Lagrange system: 
\begin{equation}\label{eq: E}
\text{div } \big( A(x, u)Du\big)=g(x, u, Du)\quad\text{in }\Omega.
\end{equation}

\noindent Under natural growth conditions on the principal part and the inhomogeneity, but \emph{without} any further restriction on the growth of the inhomogeneity (for example, via a smallness condition), we use a blow-up argument to prove that every bounded, weak solution of~\eqref{eq: E} is H\"older continuous. Since the dimension of $\Omega$ is $2$ and $u\in W^{1, 2}(\Omega;\mathbb{R}^N)$, we are in the critical setting, and hence, cannot use the Sobolev embedding theorem to deduce H\"older continuity. 

\noindent Our results are connected to a particular case of the open problem of whether all solutions (and not just extremals) of variational systems are H\"older continuous in the critical setting.
\end{abstract}
\section{Introduction}
\label{intro}

Let $\Omega\subset\mathbb{R}^n, n\geq 2$ be a bounded, Lipschitz domain. Beck and Frehse~\cite{Frehse2013} considered elliptic systems of the type:
\begin{equation}\label{eq: bfeq}
\text{div }\left(a(x, u, Du)\right)=a_0(x, u, Du)\quad\text{in }\Omega\subset\mathbb{R}^n.
\end{equation}The principal part $a$ and the inhomogeneity $a_0$ are Carath\'eodory functions, and the system satisfies the natural growth conditions, that is, 
\begin{equation}\label{eq: ngc}
\left\{\begin{array}{c c}
|a(x, z, \zeta)|&\leq K(1+|\zeta|^{p-1})\\
|a_0(x, z, \zeta)|&\leq K_0(1+|\zeta|^p)
\end{array}\right.
\end{equation}for all $(x, z, \zeta)\in \Omega\times\mathbb{R}^N\times\mathbb{R}^{Nn}$, for some $K, K_0>0$ and some fixed $p>1$. We recall what it means to weakly solve~\eqref{eq: bfeq}.

\begin{defn}\label{wksoln}
A function $u\in W^{1, p}\cap L^{\infty}(\Omega;\mathbb{R}^N)$ is called a bounded, weak solution of~\eqref{eq: bfeq} if
\begin{equation}\label{eq: weakeq}
\int_{\Omega}a(x, u, Du)\cdot D\varphi\ \mathrm{d}x=\int_{\Omega}a_0(x, u, Du)\cdot \varphi\ \mathrm{d}x
\end{equation}for all $\varphi\in C_c^{\infty}(\Omega;\mathbb{R}^N)$.
\end{defn} 
\begin{rmk}\label{rtf}
Strictly speaking, the test functions $\varphi$ in Definition~\ref{wksoln} should be in $W^{1, p}_0\cap L^{\infty}(\Omega;\mathbb{R}^N)$. However, it suffices to show~\eqref{eq: weakeq} for every $\varphi\in C_c^{\infty}(\Omega;\mathbb{R}^N)$, as~\eqref{eq: weakeq} will then hold for all $\varphi\in W^{1, p}_0\cap L^{\infty}(\Omega;\mathbb{R}^N)$ via a density argument.
\end{rmk}
\begin{rmk}If $u\in W^{1, p}(\Omega;\mathbb{R}^N)$ and satisfies~\eqref{eq: weakeq}, then we say $u$ is a weak solution of~\eqref{eq: bfeq}. Typically, we would assume that the system satisfies the controllable growth conditions, that is,
\begin{align*}
|a(x, z, \zeta)|&\leq K(1+|\zeta|^{p-1})\\
|a_0(x, z, \zeta)|&\leq K_0(1+|\zeta|^{p-1})
\end{align*}for all $(x, z, \zeta)\in \Omega\times\mathbb{R}^N\times\mathbb{R}^{Nn}$, for some $K, K_0>0$ and some fixed $p>1$.\end{rmk} 

When $p=n$, we are in the critical setting, that is, the Sobolev embedding theorem does not say whether or not $u$ is H\"older continuous.

Beck and Frehse~\cite[Thm. 1.4]{Frehse2013} demonstrated that, under zero-Dirchlet boundary conditions, there is at least one weak vector-valued solution $u:\Omega\rightarrow\mathbb{R}^N$ that is locally H\"older continuous in the critical setting. In other words, there is a $u\in W^{1, n}\cap C^{0, \gamma}_{\text{loc}}(\Omega;\mathbb{R}^N)$ for some $\gamma\in (0, 1)$ that weakly solves~\eqref{eq: bfeq}.  

It is of further interest to investigate the regularity of bounded, weak solutions of~\eqref{eq: bfeq}. A possible strategy is to consider cases depending on the particular structure of the principal part. For instance, we might restrict ourselves to diagonal systems, variational systems or non-diagonal systems. The principal part is said to be of diagonal form if 
\begin{equation*}
a^{\alpha}_i(x, z, \zeta)=\sum_{k=1}^n A_{ik}(x, z)\zeta^{\alpha}_k
\end{equation*}for $\alpha\in \{1, \dots, N\},\ i\in\{1, \dots, n\}$ and $(x, z, \zeta)\in\Omega\times\mathbb{R}^N\times\mathbb{R}^{Nn}$. Otherwise, it is of non-diagonal form. Variational systems are Euler-Lagrange systems of variational integrals
\begin{equation*}
w\mapsto \int_{\Omega}f(x, w, Dw)\ \mathrm{d}x
\end{equation*}that have a sufficiently regular integrand $f: \Omega\times\mathbb{R}^N\times\mathbb{R}^{Nn}\rightarrow\mathbb{R}$, in which case we have
\begin{equation*}
a=D_\zeta f\quad\text{and}\quad a_0=D_zf.
\end{equation*}Note that variational systems are not necessarily diagonal in general.

Aside from the natural growth conditions, one typically needs further growth restrictions of the inhomogeneity to prevent certain irregularities. For instance, 
\begin{equation}\label{eq: hmapsol}
u(x)=\frac{x}{|x|}\quad(x\in B(0, 1)\subset\mathbb{R}^3)
\end{equation} is a bounded, weak solution $u\in W^{1, 2}\cap L^{\infty}(B(0, 1);\mathbb{R}^3)$ of 
\begin{equation}\label{eq: hmapeq}
-\triangle u=|Du|^2u\quad\text{in }B(0, 1),
\end{equation} but it is discontinuous at the origin. 

It is, therefore, necessary to impose some further structure assumptions to exclude such solutions. Typically, one controls the growth of the principal part from below via an ellipticity condition:
\begin{equation*}
a(x, z, \zeta)\cdot\zeta\geq\lambda |\zeta|^n,
\end{equation*}for all $(x, z, \zeta)\in \Omega\times\mathbb{R}^N\times\mathbb{R}^{Nn}$ and some $\lambda>0$. To restrict the growth of the inhomogeneity, one can impose a one-sided condition on the inhomogeneity:
\begin{equation*}
a_0(x, z, \zeta)\cdot z\leq\lambda^{\ast}|\zeta|^n
\end{equation*}for all $(x, z, \zeta)\in\Omega\times\mathbb{R}^N\times\mathbb{R}^{Nn}$ and some $\lambda^{\ast}\in (0, \lambda)$. Alternatively, one can impose a smallness condition in terms of the $L^{\infty}$-norm of the solution itself: 
\begin{equation}\label{eq: smallnesscond}
K_0\|u\|_{\infty}<\lambda.
\end{equation} Note that the one-sided condition is a weaker condition than the smallness condition. If $\lambda=K=K_0=1$, then solutions like~\eqref{eq: hmapsol} for~\eqref{eq: hmapeq} would not be considered because the smallness and one-sided conditions are violated. 

Several regularity results are already known under stronger versions of~\eqref{eq: smallnesscond}, see, for example, the list in Hildebrandt's survey paper~\cite[p. 535]{H1982}. More specifically, the assumptions $2K_0M<\lambda$ or $\lambda^{\ast}+K_0M<\lambda$ (known as a two-sided condition), have been widely assumed by many authors for non-diagonal systems in obtaining a number of regularity results, see, for example,~\cite[p. 326]{GQ1978},~\cite[pp. 15--16]{Hamburger1998} and~\cite[Lemma 4.1]{Duz2000}. It has remained a long-standing open problem as to whether the results can still hold for $K_0M<\lambda$.

Hildebrandt and Widman considered diagonal systems and showed that under the smallness condition $K_0M<\lambda$ and when $p=n=2$, bounded, weak solutions are locally H\"older continuous~\cite[Thm. 4.1]{Hildebrandt1975}. They conjectured that the smallness condition could be weakened to a one-sided condition on the inhomogeneity without compromising the regularity upshot. Indeed, for two dimensions, Wiegner~\cite[Thm. 1]{W1981} proved that all bounded, weak solutions  of the diagonal system:
\begin{equation*}
\text{div }\left(A_{ij}(x)D_ju^{\alpha}\right)=g^{\alpha}(x, u, Du),
\end{equation*}
are H\"older continuous when the inhomogeneity satisfies a one-sided condition.

Not every non-diagonal system under the one-sided condition will guarantee that bounded, weak solutions are H\"older continuous in the critical case. Indeed, Beck and Frehse~\cite[\S 3.1]{Frehse2013} gave a counterexample. However, their counterexample is not applicable for our setting as the system does not follow the variational structure explored here. Certainly, it is an open problem whether all bounded, weak solutions of every non-diagonal system with the smallness condition~\eqref{eq: smallnesscond} are H\"older continuous in the critical setting. 

For variational systems, Morrey~\cite[Thm. 4.3]{Morrey66} proved that weak minima for variational systems are H\"older continuous in the critical case. Bounded weak local minima for quadratic functionals under diagonal coefficients, that is, functionals for which the integrand is of the form $f(x, z, \zeta)=A_{ij}(x, z)\zeta_j^{\alpha}\zeta_i^{\alpha}$, are H\"older continuous if the inhomogeneity satisfies a one-sided condition (Giaquinta and Giusti~\cite[Thm. 5.2]{GQ1982}).  Lastly, Beck and Frehse remark~\cite[p. 947]{Frehse2013}:
\begin{quote} It is an interesting, open problem whether \emph{all} solutions (such as non-extremals of the Euler equation) with smooth data are H\"older continuous, in particular for the two-dimensional case $n=p=2$.\end{quote} This paper is in response to the above remark. 

\section{Assumptions and statement of main result}
In this paper, we demonstrate regularity, in the sense of H\"older continuity, of bounded, weak vector-valued solutions $u:\Omega\rightarrow\mathbb{R}^N$ of the following elliptic system:
\begin{equation}\label{eq: maineq}
\text{div }\left(A^{\alpha\beta}_{ij}(x, u)D_ju^{\beta}\right)=\frac{1}{2}\frac{\partial A^{\gamma\beta}_{ij}}{\partial z^{\alpha}}(x, u)D_ju^{\beta}D_iu^{\gamma}\quad\text{in }\Omega\quad(A^{\alpha\beta}_{ij}=A^{\beta\alpha}_{ji}).
\end{equation} We index $\alpha, \beta$ and $\gamma$ from $1$ to $N$, while we index $i$ and $j$ from $1$ to $2$. We let $u_{x_0, r}$ denote the integral average of $u$ on the the ball $B(x_0, r)$. The principal part and the inhomogeneity are defined on $\overline{\Omega}\times\mathbb{R}^N\times\mathbb{R}^{2N}$, and we denote their arguments by $x\in\Omega, z\in\mathbb{R}^N$ and $\zeta\in\mathbb{R}^{2N}$, respectively. The system~\eqref{eq: maineq} is the Euler-Lagrange system of the following quadratic functional:
\begin{equation}\label{eq: quad}
\int_{\Omega}A^{\alpha\beta}_{ij}(x, u)D_ju^{\beta}D_iu^{\alpha}\ \mathrm{d}x. 
\end{equation}Note that the system~\eqref{eq: maineq} is in general allowed to be of non-diagonal type.

We assume that the following hypotheses are satisfied:
\begin{enumerate}
\item[(H1)] The domain $\Omega$ is an open, bounded subset of $\mathbb{R}^2$ with Lipschitz boundary.
\item[(H2)] The coefficients $A^{\alpha \beta}_{ij}(x, z)$ are smooth in $\overline{\Omega}\times\mathbb{R}^N$ and satisfy the following estimates for some $K, K_0>0$ and for all $(x, z)\in \overline{\Omega}\times\mathbb{R}^N$:
\begin{equation}\label{eq: H2}
\left|\frac{\partial A^{\gamma\beta}_{ij}}{\partial z^{\alpha}}(x, z)\right| \leq 2K_0\quad\text{and}\quad |A_{ij}^{\alpha \beta}(x, z)|\leq K.
\end{equation}Note that the estimates~\eqref{eq: H2} imply that the principal part and the inhomogeneity satisfy the natural growth conditions:
\begin{equation}\label{eq: H3}
\left|\frac{1}{2}\frac{\partial A^{\gamma\beta}_{ij}}{\partial z^{\alpha}}(x, z)\zeta_j^{\beta}\zeta_i^{\gamma}\right| \leq K_0|\zeta|^2\text{ and }|A_{ij}^{\alpha \beta}(x, z) \zeta_j^{\beta}|\leq K|\zeta|
\end{equation}for all $x\in\overline{\Omega}, z\in \mathbb{R}^N$ and $\zeta\in \mathbb{R}^{2N}$.
\item[(H3)] The principal part fulfills the ellipticity condition:
\begin{equation*}
A^{\alpha \beta}_{ij}(x, z)\zeta_j^{\beta}\zeta_i^{\alpha}\geq \lambda |\zeta|^2
\end{equation*}for some $\lambda>0$ and for all $(x, z, \zeta)\in\overline{\Omega}\times \mathbb{R}^N\times\mathbb{R}^{2N}$.
\end{enumerate}
\begin{rmk}
The diagonal version of~\eqref{eq: maineq} has connections to geometry, for instance, in the theory of harmonic mappings between Riemannian manifolds, see~\cite{ES1964}.
\end{rmk}

Our problem, which is interesting in its own right, is a particular case of the open problem, mentioned by Beck and Frehse. 
As $u\in W^{1, 2}(\Omega;\mathbb{R}^N)$ and the dimension of $\Omega$ is $2$, we are in the critical setting, and therefore, we cannot deduce H\"older continuity immediately via the Sobolev embedding theorem.

Our main result is the following regularity result for bounded, weak solutions of~\eqref{eq: maineq}:
\begin{thm}\label{regularity}
If $u\in W^{1, 2}\cap L^{\infty}(\Omega;\mathbb{R}^N)$ with $M\equiv\|u\|_{\infty}$ is a bounded, weak solution to the system~\eqref{eq: maineq} under assumptions (H1) to (H3), then for any $\gamma\in (0, 1)$ we have $u\in C^{0, \gamma}_{\text{loc}}(\Omega;\mathbb{R}^N)$.
\end{thm}

Remarkably, we arrive at the regularity result \emph{without} any further restriction on the growth of the inhomogeneity. In particular, we do not impose a smallness condition or a one-sided condition. Aside from partially resolving the open problem framed in Beck's and Frehse's paper, it extends Wiegner's result~\cite[Thm. 1]{W1981} to non-diagonal systems, it allows for Giaquinta's and Giusti's~\cite{GQ1978} result to hold true without any growth restrictions on the inhomogeneity and it also extends Giaquinta's and Giusti's~\cite{GQ1982} result to not just bounded minima but all critical points of the quadratic functional (without further growth restrictions on the inhomogeneity).
 
\section{Overview of technique}
We recall the integral characterisation of H\"older continuous functions via Campanato spaces $\mathcal{L}^{p, \mu}(\Omega;\mathbb{R}^N)$
for $\mu>n$:
\begin{equation}\label{eq: camp}
\mathcal{L}^{p, \mu}(\Omega;\mathbb{R}^N)\cong C^{0, \alpha}(\overline{\Omega};\mathbb{R}^N)
\end{equation}for some $\alpha\in (0, 1]$. Given $\gamma\in (0, 1)$, we turn our attention to proving that
\begin{equation}\label{eq: reg1}
u\in \mathcal{L}^{2, 2+2\gamma}_{\text{loc}}(\Omega;\mathbb{R}^N),
\end{equation}whence Theorem~\ref{regularity} follows by~\eqref{eq: camp} (for $\alpha=\gamma$). The first step is to obtain an energy-decay estimate on certain balls about a common centre of fixed but (discretely) shrinking radii. An iteration of this energy-decay estimate then shows~\eqref{eq: reg1}.

Establishing the energy-decay estimate, can be done directly or by contradiction. We use a blow-up method, which is an argument by contradiction, to obtain the energy-decay estimate. The technique can be traced back to De Giorgi and Almgren (cited in~\cite[p. 269]{Duz2000}) although they used it on the \emph{excess function} of the solution. The essential idea is to assume that the energy-estimate fails on a sequence of shrinking balls and then to shift and rescale each ball, that is, to blow-up each ball into the unit ball. Furthermore, we have a sequence of solutions to the corresponding sequence of systems in the unit ball. Each of the energies of the blown up solutions also violate the estimate in the unit ball. However, the blown-up solutions converge in the limit to a solution of a constant coefficient, homogeneous and elliptic system. It is known that the energy of such solutions satisfies the estimate, see, for example,~\cite[Chap. III]{GQ83}, resulting in a contradiction.

\section{Preliminaries}

\begin{defn}
Given $u\in W^{1, 2}(\Omega;\mathbb{R}^N)$, we define its \emph{energy} on a given ball $B(x_0, r)\subset\subset\Omega$ as:
\begin{equation}\label{eq: energy}
\Phi(x_0, r)\equiv \int_{B(x_0, r)}|Du|^2\ \mathrm{d}x.
\end{equation}
\end{defn}

We state the following form of Poincar\'e's inequality on balls and refer the reader, for example, to~\cite[Appendix 1 \S 3]{Chen98} for a proof.
\begin{thm}
Let $B(x, r)$ be a ball in $\mathbb{R}^n$ with radius $r$. If $u\in W^{1, p}(B_r)$ for some $p\in[1, \infty)$, then
\begin{equation}\label{eq: P}
\int_{B(x, r)}|u-u_{x, r}|^p\ \mathrm{d}y\leq C(n, p)r^p\int_{B(x, r)}|Du|^p\ \mathrm{d}y.
\end{equation}
\end{thm}

If we have a solution of a constant coefficient, homogeneous and elliptic system, then its gradient will also solve the same system. We can then apply the following result from Giaquinta's book~\cite[Chap. III, Thm. 2.1]{GQ83} immediately to the gradient in the unit ball $B(0, 1)\subset\mathbb{R}^2$.

\begin{lem}\label{blowupunitball}
Let $b^{\alpha\beta}_{ij}$ be constants for $\alpha, \beta\in\{1\dots, N\}$ and $i,j\in\{1, 2\}$ satisfying:
\begin{equation*}
L|\zeta|^2\geq b^{\alpha\beta}_{ij}\zeta^{\beta}_j\zeta^{\alpha}_i\geq\lambda|\zeta|^2
\end{equation*}for all $\zeta\in\mathbb{R}^{2N}$ and some $\lambda, L>0$. If $u\in W^{1, 2}\cap L^{\infty}(\Omega;\mathbb{R}^N)$ satisfies:
\begin{equation*}
\int_{B(0, 1)}b^{\alpha\beta}_{ij}D_ju^{\beta}D_i\varphi^{\alpha}\ \mathrm{d}x=0
\end{equation*}for every $\varphi\in C_c^{\infty}(B(0, \rho);\mathbb{R}^N)$, then
\begin{equation*}
\Phi(0, \rho)\leq c_0\rho^2\Phi(0, 1)
\end{equation*}for all $\rho\in (0, 1)$ and some $c_0=c_0(N, \lambda, L)>1$. 
\end{lem}
\section{The energy-decay estimate}
Henceforth, we let $u\in W^{1, 2}\cap L^{\infty}(\Omega;\mathbb{R}^N)$ with $\|u\|_{L^{\infty}(\Omega;\mathbb{R}^N)}\equiv M$ be a bounded, weak solution of the system~\eqref{eq: maineq} under the assumptions (H1) to (H3). We also let $\Phi(x_0, r)$ as defined in~\eqref{eq: energy} denote its energy on the ball $B(x_0, r)\subset\subset\Omega$.

\begin{prop}\label{blowup}	
Let $\tau\in (0, 1)$ be given. Then there is an $r_0=r_0(\tau, N, K,\lambda)>0$ and an $\varepsilon_0=\varepsilon_0(\tau, N, K,\lambda)>0$ such that if 
\begin{equation*}
\Phi(x_0, r)<\varepsilon_0^2
\end{equation*}for some $B(x_0, r)\subset\subset\Omega$ and some $r\in (0, r_0)$, then 
\begin{equation*}
\Phi(x_0, \tau r)\leq 2c_0\tau^2\Phi(x_0, r).
\end{equation*}Here, $c_0=c_0(N, K, \lambda)>1$ is the constant in Lemma~\ref{blowupunitball}.
\end{prop}
\begin{proof}
If the proposition were false for some $\tau\in (0, 1)$, then we can seek out a sequence of balls $B(x_m, r_m)\subset\subset\Omega$ such that $r_m\searrow 0$ and
\begin{equation}\label{eq: rev1d}
\Phi(x_m, r_m)\equiv\varepsilon_m^2\rightarrow 0
\end{equation} as $m\rightarrow\infty$, but
\begin{equation}\label{eq: rev2d}
\Phi(x_m, \tau r_m)>2c_0\tau^2\varepsilon_m^2
\end{equation}for each $m\in\mathbb{N}$.

We shift and rescale, that is, we blow-up each ball $B(x_m,r_m)$ into the unit ball $B\equiv B(0, 1)$ by defining 
\begin{equation*}
y\equiv \frac{x-x_m}{r_m}\quad(m\in\mathbb{N}, x\in B(x_m, r_m)).
\end{equation*}Note that $y\in B$. We also let
\begin{equation*}
v_m(y)\equiv\frac{u(x_m+r_my)-u_{x_m, r_m}}{\varepsilon_m}\quad(y\in B, \ m\in\mathbb{N}).
\end{equation*}By a change of coordinates from $B(x_m, r_m)$ to $B$, we immediately have 
\begin{equation}\label{eq: pi}
(v_m)_{0, 1}=0
\end{equation}and
\begin{equation}\label{eq: pii}
Dv_m(y)=\frac{r_m}{\varepsilon_m}Du(x)
\end{equation}for each $v_m$.

For each $v_m$, we define its energy in $B(z, r)\subseteq B$ as
\begin{equation*}
\Psi_m(z, r)\equiv\int_{B(z, r)}|Dv_m|^2\ \mathrm{d}y.
\end{equation*}Then following a change of coordinates from $B(x_m, r_m)$ to $B$, we deduce that
\begin{equation}\label{eq: piii}
\Psi_{m}(0, 1)=\int_{B(0, 1)}|Dv_m(y)|^2\ \mathrm{d}y=1
\end{equation}from~\eqref{eq: rev1d} and
\begin{equation}\label{eq: piv}
\Psi_{m}(0, \tau)=\int_{B(0, \tau)}|Dv_m(y)|^2\ \mathrm{d}y>2c_0\tau^2
\end{equation}from~\eqref{eq: rev2d}, respectively. 

By Poincar\'e's inequality~\eqref{eq: P},~\eqref{eq: pi} and~\eqref{eq: piii},
\begin{equation*}
\int_{B(0, 1)}\varepsilon_m^2|v_m(y)|^2\ \mathrm{d}y\leq C\varepsilon_m^2\int_{B(0, 1)}|Dv_m(y)|^2\ \mathrm{d}y=C\varepsilon_m^2,
\end{equation*}and consequently,
\begin{equation}\label{eq: subseq1}
\lim_{m\rightarrow\infty}\int_{B(0, 1)}\varepsilon_m^2|v_m(y)|^2\ \mathrm{d}y=0.
\end{equation}As $\overline{\Omega}$ is bounded, there is some ball $B(0, R_0)\subset\mathbb{R}^2$ such that $\Omega\subset B(0, R_0)$, and consequently,
\begin{equation}\label{eq: subseq2}
|x_m|\leq R_0 \quad(m\in\mathbb{N}).
\end{equation}Given that $\|u\|_{\infty}=M$, we also have
\begin{equation}\label{eq: subseq3}
|u_{x_m, r_m}|\leq M\quad(m\in\mathbb{N}).
\end{equation}Finally, we recall that
\begin{equation}\label{eq: subseq4}
\lim_{m\rightarrow\infty}r_m=0.
\end{equation}

It follows from~\eqref{eq: subseq1},~\eqref{eq: subseq2} and~\eqref{eq: subseq3} that, by passing to a subsequence and relabelling, if necessary, as $m\rightarrow\infty$ we have
\begin{equation}\label{eq: subseq5}
\begin{array}{l l}
\varepsilon_mv_m\rightarrow 0 & \quad\text{a.e. in }B\\
(x_m, u_{x_m, r_m})\rightarrow (x_0, u_0) &\quad\text{for some } (x_0, u_0)\in \overline{\Omega}\times\mathbb{R}^N.
\end{array}
\end{equation}Then~\eqref{eq: subseq5},~\eqref{eq: subseq4} and (H2) imply that
\begin{equation}\label{eq: constconv2}
A^{\alpha\beta}_{ij}\big(x_m+r_my, u_{x_m, r_m}+\varepsilon_mv_m(y)\big)\rightarrow b^{\alpha\beta}_{ij}\text{ a.e. in }  B(0, 1)
\end{equation}as $m\rightarrow\infty$ for some constants $b^{\alpha\beta}_{ij}$. These constant coefficients satisfy the ellipticity condition (H3) because the coefficients $A^{\alpha\beta}_{ij}(x_m+r_my, u_{x_m, r_m}+\varepsilon_mv_m(y))$ do for each $m\in\mathbb{N}$ and are smooth in $\overline{\Omega}\times\mathbb{R}^N\times\mathbb{R}^{2N}$.

Poincar\'e's inequality~\eqref{eq: P},~\eqref{eq: pi} and~\eqref{eq: piii} also imply that the sequence $\{v_m\}_{m=1}^{\infty}$ is uniformly bounded in $W^{1, 2}(B(0, 1);\mathbb{R}^N)$ since
\begin{equation*}
\int_{B(0, 1)}|v_m(y)|^2\ \mathrm{d}y\leq C\int_{B(0, 1)}|Dv_m(y)|^2\ \mathrm{d}y=C
\end{equation*}for all $m\in\mathbb{N}$. Therefore, we can pass to a subsequence (that we still denote as $\{v_m\}_{m=1}^{\infty}$) such that as $m\rightarrow\infty$, 
 \begin{equation}\label{eq: usefulconv}
\left\{\begin{array}{l l}
v_{m}\rightarrow v & \text{ in } L^2(B(0, 1);\mathbb{R}^N)\\
Dv_{m}\rightharpoonup Dv& \text{ in }L^2(B(0, 1);\mathbb{R}^{2N})
\end{array}\right.
\end{equation}for some $v\in W^{1, 2}(B(0, 1);\mathbb{R}^N)$. Let the energy of $v$ in the ball $B(\sigma, \rho)\subset B$ be given as:
\begin{equation*}
\Psi(\sigma, \rho)\equiv \int_{B(\sigma, \rho)}|Dv|^2\ \mathrm{d}y.
\end{equation*}Via the lower semicontinuity of the norm with respect to the weak convergence in~\eqref{eq: usefulconv}, we arrive at the following bound on $Dv$:
\begin{equation*}
\|Dv\|_{L^2(B(0, 1))}^2\leq \liminf_{m\rightarrow\infty}\|Dv_m\|_{L^2(B(0, 1))}^2=1.
\end{equation*}This bound implies that
\begin{equation}\label{eq: contra3}
\Psi(0, 1)\leq 1.
\end{equation}

Given any $\varphi\in C_c^{\infty}(B(0, 1);\mathbb{R}^N)$, suppose we can demonstrate that
\begin{equation}\label{eq: cool}
\int_{B(0, 1)}b^{\alpha\beta}_{ij}D_jv^{\beta}D_i\varphi^{\alpha}\ \mathrm{d}y=0.
\end{equation} Furthermore, suppose that we can improve the weak convergence of the gradients to strong convergence, that is,
\begin{equation}\label{eq: usefulconv2}
Dv_m\rightarrow Dv\quad \text{in }L^2_{\text{loc}}(B(0, 1);\mathbb{R}^{2N})
\end{equation}as $m\rightarrow\infty$.
 
Then via Lemma~\ref{blowupunitball},~\eqref{eq: usefulconv2},~\eqref{eq: piv} and~\eqref{eq: contra3} we have
\begin{equation}\label{eq: contraiv}
c_0\tau^2\Psi(0, 1)\geq\Psi(0, \tau)=\lim_{m\rightarrow\infty}\Psi_{m}(0, \tau)\geq2c_0\tau^2\geq2c_0\tau^2\Psi(0, 1).
\end{equation}If $\Psi(0, 1)=0$, then $\Psi(0, \rho)=0$ for each $\rho\in (0, 1)$ because the energy is non-negative. In particular, we would have $\Psi(0, \tau)=0$. However, $\Psi(0, \tau)>0$ by~\eqref{eq: piv} and~\eqref{eq: usefulconv2}. Therefore, $\Psi(0, 1)>0$, and we arrive at the contradiction $1\geq 2$ from~\eqref{eq: contraiv} thereby concluding the proof.
\end{proof}

\begin{rmk}\label{hindrance}
Suppose that we wish to generalise Theorem~\ref{regularity} to higher dimensions. Assume that we have appropriate growth conditions so that a weak solution $u\in W^{1, n}\cap L^{\infty}$ to~\eqref{eq: maineq} makes sense. Naturally, we would define its energy on a ball $B(x, r)\subset\subset\Omega$ as
\begin{equation*}
\Phi^{\ast}(x, r)\equiv\int_{B(x, r)}|Du|^n\ \mathrm{d}x
\end{equation*}The analogue of Proposition~\ref{blowup} in higher dimensions would roughly read something akin to: given a sufficiently small $\tau\in (0, 1)$, there exists an $(\varepsilon_0, r_0)\in (0, \infty)\times (0, \infty)$ such that if for some small enough ball $\Phi^{\ast}(x, r)<\varepsilon_m^n$, then we must have $\Phi(x_0, \tau r)\leq 2c_0^{\ast}\tau^n\Phi(x_0, r)$ for some $c_0^{\ast}=c_0^{\ast}(n, N, K, \lambda)>1$. In any case, we initiate the blow-up argument and expect that 
\begin{equation*}
\int_{B(0, 1)}|Dv_m|^n\ \mathrm{d}y=1
\end{equation*}analogous to~\eqref{eq: piii}. However, as we would have
\begin{equation*}
\int_{B(x_m, r_m)}|Du|^n\ \mathrm{d}x=\varepsilon_m^n,
\end{equation*}then upon rescaling we arrive at
\begin{equation*}
\int_{B(0, 1)}|Dv_m|^n\ \mathrm{d}y=r_m^{2-n}\rightarrow\infty\quad\text{as}\quad m\rightarrow\infty.
\end{equation*}This is not ideal since we do not subsequently have weak compactness of the sequence $\{v_m\}_{m\in\mathbb{N}}\subset W^{1, n}(\Omega;\mathbb{R}^N)$. Without weak compactness we cannot conclude the existence of a limit solution in the unit ball nor can we obtain the analogues of~\eqref{eq: cool} and~\eqref{eq: usefulconv2}. This points to a major obstacle in proving regularity in higher dimensions via the blow-up argument on the energy of the solution as well as the particulars of the argument to the two-dimensional setting.
\end{rmk}
It remains to prove~\eqref{eq: cool} and~\eqref{eq: usefulconv2}. These are formulated as Lemma~\ref{homogsys} and Lemma~\ref{wktostrong}, respectively. For brevity in the forthcoming sections, we let 
\begin{equation*}
A^{\alpha\beta}_{ij}(x_m+r_my, u_{x_m, r_m}+\varepsilon_mv_m(y))\equiv A^{\alpha\beta}_{ij,m}(y)
\end{equation*}and
\begin{equation*}
\frac{\partial A^{\gamma\beta}_{ij}}{\partial z^{\alpha}}(x_m+r_my, u_{x_m, r_m}+\varepsilon_mv_m(y))\equiv\frac{\partial A^{\gamma\beta}_{ij,m}}{\partial z^{\alpha}}(y),
\end{equation*}respectively.
\section{Convergence of the rescaled solutions to a linear system}
The goal of this section is to prove~\eqref{eq: cool}.
\begin{lem}\label{constconv}
For each $\phi\in L^2(B(0, 1);\mathbb{R}^{2N})$, we have
\begin{equation}\label{eq: constconv}
\int_{B(0, 1)}A^{\alpha \beta}_{ij, m}(y)D_jv_m^{\beta}(y)\phi_{i}^{\alpha}(y)\ \mathrm{d}y\rightarrow \int_{B(0, 1)}b^{\alpha \beta}_{ij}D_jv^{\beta}(y)\phi_i^{\alpha}(y)\ \mathrm{d}y,
\end{equation}as $m\rightarrow\infty$.
\end{lem}
\begin{proof}
Given $\phi\in L^2(B(0, 1);\mathbb{R}^{2N})$, we arrive at the following estimate via the triangle inequality, H\"older's inequality and~\eqref{eq: piii}:
\begin{align}\label{eq: rhs}
&\left|\int_{B(0, 1)}[A^{\alpha \beta}_{ij,m}(y)D_jv_m^{\beta}(y)-b^{\alpha\beta}_{ij}D_jv^{\beta}(y)]\phi_i^{\alpha}(y)\ \mathrm{d}y\right|\nonumber\\
&\leq\int_{B(0,\rho)}|A^{\alpha\beta}_{ij,m}(y)-b^{\alpha\beta}_{ij}||D_jv_m^{\beta}(y)||\phi_i^{\alpha}(y)|\ \mathrm{d}y\nonumber\\
&\quad+\left|\int_{B(0, \rho)}b^{\alpha\beta}_{ij}[D_jv_m^{\beta}(y)-D_jv^{\beta}(y)]\phi_i^{\alpha}(y)\ \mathrm{d}y\right|\nonumber\\
&\leq \left(\int_{B(0,1)}|A^{\alpha\beta}_{ij,m}(y)-b^{\alpha\beta}_{ij}|^2|\phi(y)|^2\ \mathrm{d}y\right)^{\frac{1}{2}}\nonumber\\
&\quad+\left|\int_{B(0, \rho)}b^{\alpha\beta}_{ij}[D_jv_m^{\beta}(y)-D_jv^{\beta}(y)]\phi_i^{\alpha}(y)\ \mathrm{d}y\right|.
\end{align}As $m\rightarrow \infty$, we see that 
\begin{equation}\label{eq: rhs1}
\int_{B(0,1)}|A^{\alpha\beta}_{ij,m}(y)-b^{\alpha\beta}_{ij}|^2|\phi(y)|^2\ \mathrm{d}y\rightarrow 0
\end{equation}by (H2) and~\eqref{eq: constconv2} and
\begin{equation}\label{eq: rhs2}
\int_{B(0, \rho)}b^{\alpha\beta}_{ij}[D_jv_m^{\beta}(y)-D_jv^{\beta}(y)]\phi_i^{\alpha}(y)\ \mathrm{d}y\rightarrow 0
\end{equation}by~\eqref{eq: usefulconv}. Therefore,~\eqref{eq: rhs1} and~\eqref{eq: rhs2} imply that the right-hand side of~\eqref{eq: rhs} vanishes as $m\rightarrow\infty$, and consequently,
\begin{equation*}
\lim_{m\rightarrow\infty}\int_{B(0, 1)}A^{\alpha \beta}_{ij,m}(y)D_jv_m^{\beta}(y)\phi_i^{\alpha}(y)\ \mathrm{d}y=\int_{B(0, 1)}b^{\alpha\beta}_{ij}D_jv^{\beta}(y)\phi_i^{\alpha}(y)\ \mathrm{d}y.
\end{equation*}
\end{proof}

\begin{lem}\label{homogsys}
The function $v\in W^{1, 2}(B;\mathbb{R}^N)$ weakly solves the linear system
\begin{equation*}
\text{div }(b^{\alpha\beta}_{ij}D_jv^{\beta})=0.
\end{equation*} 
\end{lem}
\begin{proof}
Given any ball $B(x_m, r_m)\subset\subset\Omega$ and any test function $\tilde{\varphi}\in C_c^{\infty}(B(x_m, r_m);\mathbb{R}^N)$, we know that $u$ satisfies the following equation:
\begin{align*}
I&\equiv\int_{B(x_m, r_m)}A^{\alpha\beta}_{ij}(x, u(x))D_ju^{\beta}(x)D_i\tilde{\varphi}^{\alpha}(x)\ \mathrm{d}x\\
&=-\int_{B(x_m, r_m)}\frac{1}{2}\frac{\partial A^{\gamma\beta}_{ij}}{\partial z^{\alpha}}(x, u(x))D_ju^{\beta}(x)D_iu^{\gamma}(x) \tilde{\varphi}^{\alpha}(x)\ \mathrm{d}x\equiv II.\\
\end{align*}By rescaling from $B(x_m, r_m)$ to $B$, we find that
\begin{equation*}
I=\varepsilon_m\int_{B(0, 1)}A^{\alpha \beta}_{ij,m}(y)D_jv_m^{\beta}(y)D_i\tilde{\varphi}^{\alpha}(x_m+r_my)\ \mathrm{d}y
\end{equation*}and
\begin{equation*}
II=-\varepsilon_m^2\int_{B(0, 1)}\frac{1}{2}\frac{\partial A^{\gamma\beta}_{ij,m}}{\partial z^{\alpha}}(y)D_jv_m^{\beta}(y)D_iv_m^{\gamma}(y)\tilde{\varphi}^{\alpha}(x_m+r_my)\ \mathrm{d}y. 
\end{equation*}Writing $\varphi_m(y)$ for $\tilde{\varphi}(x_m +r_my)$ and equating $I$ to $II$, we see that
\begin{align}\label{eq: yeq}
&\int_{B(0, 1)}A^{\alpha \beta}_{ij,m}(y)D_jv_m^{\beta}(y)D_i\varphi_m^{\alpha}(y)\ \mathrm{d}y\nonumber\\
&\qquad=-\varepsilon_m\int_{B(0, 1)}\frac{1}{2}\frac{\partial A^{\gamma\beta}_{ij,m}}{\partial z^{\alpha}}(y)D_jv_m^{\beta}(y)D_iv_m^{\gamma}(y)\varphi_m^{\alpha}(y)\ \mathrm{d}y.
\end{align}

Given any $\varphi\in C_c^{\infty}(B(0, 1);\mathbb{R}^N)$, we arrive at the following estimate for the left-hand side of~\eqref{eq: yeq} by (H2) and~\eqref{eq: piii}:\begin{equation}
\left|\int_{B(0, 1)}A^{\alpha \beta}_{ij,m}(y)D_jv_m^{\beta}(y)D_i\varphi^{\alpha}(y)\ \mathrm{d}y\right|\leq \varepsilon_mK_0\|Dv_m\|_{L^2(B)}^2\|\varphi\|_{\infty}\leq\varepsilon_mK_0\|\varphi\|_{\infty}.
\end{equation} Thus,
\begin{equation}\label{eq: one}
\lim_{m\rightarrow\infty}\int_{B(0, 1)}A^{\alpha \beta}_{ij,m}(y)D_jv_m^{\beta}(y)D_i\varphi^{\alpha}(y)\ \mathrm{d}y=0\quad(\varphi\in C_c^{\infty}(B(0, 1);\mathbb{R}^N))
\end{equation}because $\varepsilon_m\rightarrow 0$ as $m\rightarrow 0$.

Given $\varphi\in C_c^{\infty}(B(0, 1);\mathbb{R}^N)$, we also recall~\eqref{eq: constconv}: 
\begin{equation}\label{eq: two}
\lim_{m\rightarrow\infty}\int_{B(0, 1)}A^{\alpha \beta}_{ij,m}(y)D_jv_m^{\beta}(y)D_i\varphi^{\alpha}(y)\ \mathrm{d}y=\int_{B(0, 1)}b^{\alpha\beta}_{ij}D_jv^{\beta}D_i\varphi^{\alpha}\ \mathrm{d}y.
\end{equation}
 
Equating~\eqref{eq: one} and~\eqref{eq: two} concludes our proof.
\end{proof} 

\section{From weak to strong convergence of the rescaled gradients}

The proof of~\eqref{eq: usefulconv2} is described in this section.

\begin{lem}\label{superior}
For almost every $r\in (0, 1)$, 
\begin{equation}\label{eq: limsupeq}
\limsup_{m\rightarrow\infty}\int_{B(0, r)}A^{\alpha\beta}_{ij,m}(y)D_jv_m^{\beta}(y)D_iv_m^{\alpha}(y)\ \mathrm{d}y\leq\int_{B(0, r)}b^{\alpha\beta}_{ij}D_jv^{\beta}(y)D_iv^{\alpha}(y)\ \mathrm{d}y
\end{equation}
\end{lem}
\begin{proof}
We will write $B$ instead of $B(0, 1)$ for brevity. For each $m\in\mathbb{N}$, we define the following sequence of measures:
\begin{equation}\label{eq: meas}
\mu_m(E)\equiv\inf\left\{\int_{G}|Dv_m|^2\ \mathrm{d}y\ \big|\ E\subset G,\ G\subset B\text{ is Borel}\right\}\quad (E\subset B).
\end{equation} By~\eqref{eq: piii}, the measures $\mu_m$ are uniformly bounded in the space of finite Radon measures on $B$. Hence, there is a finite Radon measure $\mu$ on $B$ such that upon passing to a subsequence, if necessary, we have that
\begin{equation}\label{eq: supconv}
\limsup_{m\rightarrow\infty}\mu_m(\mathcal{K})\leq \mu(\mathcal{K})
\end{equation}for each compact $\mathcal{K}\subset B$. We refer the reader to~\cite[\S 1.9]{GE92}, for instance, concerning the basic theory of weak convergence and compactness of Radon measures including the proof of~\eqref{eq: supconv}.

Since $\mu(B)<\infty$ we have that $\mu(\partial B_r)=0$ for all but countably many $r\in(0, 1)$. Such a result is a consequence of finite measures and can be found, for example, in~\cite[Prop. 1.15]{Fo07}. 

Let $r\in (0, 1)$ be such that 
\begin{equation}\label{eq: mu}
\mu(\partial B_r)=0
\end{equation} and let $s\in (0, r)$. We let $\eta\in C_c^{\infty}(B(0, r);\mathbb{R})$ be a smooth cutoff function such that $\eta\equiv 1\text{ on }B_s$.

Next, fix $\sigma\in (0, \infty)$ and define $\xi_{\sigma}:\mathbb{R}\rightarrow \mathbb{R}$ as follows:
\begin{equation*}
\xi_{\sigma}(t)=\left\{\begin{array}{l r}
\sigma & t>\sigma\\
t & -\sigma\leq t\leq \sigma\\
-\sigma & t<-\sigma.
\end{array}\right.
\end{equation*}The function $\xi_{\sigma}$ is not differentiable only on the set $\{\sigma, -\sigma\}$. However, as $|\{\sigma, -\sigma\}|=0$, the function $\xi_{\sigma}$ has a unique weak derivative that agrees with the classical derivative of $\xi_{\sigma}$ upto the set $\{\sigma, -\sigma\}$. Thus, 
\begin{equation}\label{eq: xideriv}
\xi_{\sigma}'(t)=\left\{\begin{array}{l r}
0 & t>\sigma\\
1 & -\sigma<t<\sigma\\
0 & t<-\sigma.
\end{array}\right.
\end{equation}is the weak derivative of $\xi_{\sigma}$. 

Finally, for each $m\in \mathbb{N}$ we consider the following test function:
\begin{equation*}
\varphi_{m}^{\alpha}(y)\equiv \eta(y)\xi_{\sigma}(v_m^{\alpha}(y)-v^{\alpha}(y))\quad (y\in B, \alpha\in\{1, \dots, N\}).
\end{equation*}Note that $\varphi_m\in W^{1, 2}_0\cap L^{\infty}(B;\mathbb{R}^N)$ for each $m\in\mathbb{R}^N$. By Remark~\ref{rtf}, it is, therefore, a suitable test function to use in~\eqref{eq: yeq} as for each $m\in\mathbb{N}$, $v_m\in W^{1, 2}\cap L^{\infty}(B;\mathbb{R}^N)$ is a bounded, weak solution to
\begin{equation*}
\text{div }(A^{\alpha\beta}_{ij, m}D_jv_m^{\beta})=\varepsilon_m\frac{1}{2}\frac{\partial A^{\gamma\beta}_{ij, m}}{\partial z^{\alpha}}D_jv_m^{\beta}D_iv_m^{\gamma}\quad\text{in } B.
\end{equation*} 

Since $\xi_{\sigma}'\in L^{\infty}(\mathbb{R})$ and $(v_m^{\alpha}-v^{\alpha})$ has a weak derivative in $B$ for each $m\in\mathbb{N}$ and $\alpha\in\{1, \dots, N\}$, we have $\xi_{\sigma}(v_m^{\alpha}-v^{\alpha})$ has a weak derivative in $B$ for each $m\in\mathbb{N}$ and $\alpha\in\{1, \dots, N\}$ given by
\begin{equation}\label{eq: kwanza}
D_i[\xi_{\sigma}(v_m^{\alpha}-v^{\alpha})]=\left\{\begin{array}{l r}
\xi_{\sigma}'(v_m^{\alpha}-v^{\alpha})(D_iv_m^{\alpha}-D_iv^{\alpha}) &\quad\text{if } (v_m^{\alpha}-v^{\alpha})\notin\{\sigma,-\sigma\}\\
0 & \quad\text{if }(v_m^{\alpha}-v^{\alpha})\in\{\sigma, -\sigma\}
\end{array}\right.
\end{equation}for a.e. $y\in B$, see, for example,~\cite[Thm. 7.8]{GT83}. Clearly, $|D[\xi_{\sigma}(v_m-v)]|\in L^2(B)$ for each $m\in\mathbb{N}$ as $v_m-v\in W^{1, 2}(B;\mathbb{R}^N)$ for each $m\in\mathbb{N}$.

As each $v_m\in L^1(B;\mathbb{R}^N)$, it follows that $|v_m|: B\rightarrow\mathbb{R}^{\ast}$ is measurable and $|v_m|<\infty$ a.e. in $B$. A measurable function must be defined at least a.e. in $B$ so that one can extend the function to the whole ball $B$ without affecting its measurability or altering its equivalence class under the `almost everywhere' equivalence relation. Subsequently, one then works with the extension. Thus, every $v_m$ is defined a.e. in $B$. Therefore,~\eqref{eq: kwanza} only fails on the null set $\{|v_m|=\infty\}$ or where $v_m$ is not defined (which is also a null set).  

By~\eqref{eq: xideriv}, we note that
\begin{equation}\label{eq: halafu}
\xi_{\sigma}'(v_m^{\alpha}-v^{\alpha})=\left\{\begin{array}{l l}
0 & \quad\text{if }v_m^{\alpha}-v^{\alpha}>\sigma\\
1 & \quad\text{if }-\sigma<v_m^{\alpha}-v^{\alpha}<\sigma\\
0 & \quad\text{if }v_m^{\alpha}-v^{\alpha}<-\sigma
\end{array}\right.
\end{equation}for each $m\in\mathbb{N}$ and $\alpha\in\{1, \dots, N\}$. For each $\alpha\in \{1, \dots, N\}, \ m\in\mathbb{N}$ and $\sigma\in (0, \infty)$, we define
\begin{equation}\label{eq: bigtheta}
\Theta_{\sigma, m}^{\alpha}\equiv\{y\in B\ |\ |v_m^{\alpha}(y)-v^{\alpha}(y)|<\sigma\}.
\end{equation}Therefore, with~\eqref{eq: halafu} and~\eqref{eq: bigtheta} we can rewrite~\eqref{eq: kwanza} as
\begin{equation}\label{eq: kwanza2}
D_i[\xi_{\sigma}(v_m^{\alpha}-v^{\alpha})]=(D_iv_m^{\alpha}-D_iv^{\alpha})\chi_{\Theta_{\sigma, m}^{\alpha}}\quad\text{a.e. in }B.
\end{equation}

Concurrently, since $\eta(\cdot)\in C_c^{\infty}(B)$ and $\xi_{\sigma}(v_m^{\alpha}(\cdot)-v^{\alpha}(\cdot))\in W^{1, 2}(B)$ for each $m\in\mathbb{N}$ and $\alpha\in\{1, \dots, N\}$, then $(\eta(\cdot)\xi_{\sigma}(v_m^{\alpha}(\cdot)-v^{\alpha}(\cdot)
))\in W^{1, 2}(B)$. Moreover, we have 
\begin{equation*}
D_i\varphi_m^{\alpha}\equiv \eta(D_iv_m^{\alpha}-D_iv^{\alpha})\chi_{\Theta_{\sigma, m}^{\alpha}}+D_i\eta\xi_{\sigma}(v_m^{\alpha}-v^{\alpha})\quad\text{a.e. in }B
\end{equation*}by the product rule for weak derivatives, see, for example,~\cite[p. 261]{Evans}, and~\eqref{eq: kwanza2}.

We then substitute $\varphi_m$ into~\eqref{eq: yeq} to deduce that
\begin{align}\label{eq: troublesome}
&\int_{B_r}\eta(y) A^{\alpha\beta}_{ij,m}(y)D_jv_m^{\beta}(y)D_iv_m^{\alpha}(y)\chi_{\Theta_{\sigma, m}^{\alpha}}(y)\ \mathrm{d}y\nonumber\\
&\leq\int_{B_r}A^{\alpha\beta}_{ij,m}(y)D_jv_m^{\beta}(y)D_iv^{\alpha}(y)\eta(y)\chi_{\Theta_{\sigma, m}^{\alpha}}(y)\ \mathrm{d}y\nonumber\\
&+\int_{B_r\setminus B_s}|A^{\alpha\beta}_{ij, m}(y)||D_jv_m^{\beta}(y)||D\eta(y)||\xi_{\sigma}(v_m^{\alpha}(y)-v^{\alpha}(y))|\ \mathrm{d}y\nonumber\\
&+\varepsilon_m\int_{B_r}\left|\frac{1}{2}\frac{\partial A^{\gamma \beta}_{ij,m}}{\partial z^{\alpha}}(y)D_jv_m^{\beta}(y)D_iv_m^{\gamma}(y)\eta(y)\xi_{\sigma}(v_m^{\alpha}(y)-v^{\alpha}(y))\right|\ \mathrm{d}y\nonumber\\
&\equiv I+II+III.
\end{align} By (H3) and the fact that $\eta\geq 0$, the left-hand side of~\eqref{eq: troublesome} is non-negative, and consequently, we can write
\begin{equation}\label{eq: troublesome1}
\int_{B_s}A^{\alpha\beta}_{ij,m}(y)D_jv_m^{\beta}(y)D_iv_m^{\alpha}(y)\chi_{\Theta_{\sigma, m}^{\alpha}}(y)\ \mathrm{d}y\leq I+II+III.
\end{equation}In particular, we note that the limit superior of the terms on the right-hand side of~\eqref{eq: troublesome1} are non-negative.

Following the argument to arrive at the estimate~\eqref{eq: rhs} with the additional fact that $\eta\chi_{\Theta_{\sigma, m}^{\alpha}}\leq 1$ for all $m\in\mathbb{N}$, we have
\begin{align}\label{eq: RHS}
&\left|\int_{B(0, 1)}[A^{\alpha \beta}_{ij,m}(y)D_jv_m^{\beta}(y)-b^{\alpha\beta}_{ij}D_jv^{\beta}(y)]\eta(y)\chi_{\Theta_{\sigma, m}^{\alpha}}(y)D_iv^{\alpha}(y)\ \mathrm{d}y\right|\nonumber\\
&\leq\int_{B(0,\rho)}|A^{\alpha\beta}_{ij,m}(y)-b^{\alpha\beta}_{ij}||D_jv_m^{\beta}(y)||D_iv^{\alpha}(y)|\ \mathrm{d}y\nonumber\\
&\quad+\left|\int_{B(0, \rho)}b^{\alpha\beta}_{ij}[D_jv_m^{\beta}(y)-D_jv^{\beta}(y)]D_iv^{\alpha}(y)\ \mathrm{d}y\right|\nonumber\\
&\leq \left(\int_{B(0,1)}|A^{\alpha\beta}_{ij,m}(y)-b^{\alpha\beta}_{ij}|^2|Dv(y)|^2\ \mathrm{d}y\right)^{\frac{1}{2}}\nonumber\\
&\quad+\left|\int_{B(0, \rho)}b^{\alpha\beta}_{ij}[D_jv_m^{\beta}(y)-D_jv^{\beta}(y)]D_iv^{\alpha}(y)\ \mathrm{d}y\right|,
\end{align}and consequently, the right-hand side of~\eqref{eq: RHS} vanishes as $m\rightarrow\infty$ by~\eqref{eq: rhs1} and~\eqref{eq: rhs2} with $\varphi=Dv$. In other words,
\begin{align*}
\lim_{m\rightarrow\infty}&\int_{B_r}A^{\alpha\beta}_{ij,m}(y)D_jv_m^{\beta}(y)D_iv^{\alpha}(y)\eta(y)\chi_{\Theta_{\sigma, m}^{\alpha}}(y)\ \mathrm{d}y\\
&=\int_{B_r}b^{\alpha\beta}_{ij}D_jv^{\beta}(y)D_iv^{\alpha}(y)\eta(y)\chi_{\Theta_{\sigma, m}^{\alpha}}(y)\ \mathrm{d}y, 
\end{align*}and therefore,
\begin{align}\label{eq: I}
\limsup_{m\rightarrow\infty} I&=\lim_{m\rightarrow\infty} I\nonumber\\
&=\lim_{m\rightarrow\infty}\int_{B_r}A^{\alpha\beta}_{ij,m}(y)D_jv_m^{\beta}(y)D_iv^{\alpha}(y)\eta(y)\chi_{\Theta_{\sigma, m}^{\alpha}}(y)\ \mathrm{d}y\nonumber\\
&=\int_{B_r}b^{\alpha\beta}_{ij}D_jv^{\beta}D_iv^{\alpha}\eta\chi_{\Theta_{\sigma, m}^{\alpha}}\ \mathrm{d}y\nonumber\\
&\leq \int_{B_r}b^{\alpha\beta}_{ij}D_jv^{\beta}D_iv^{\alpha}\ \mathrm{d}y.
\end{align}

We use (H2) and Young's inequality on II to deduce
\begin{equation*}
II\leq \frac{K^2}{2}\mu_m(\overline{B_r\setminus B_s})+\frac{1}{2}\|D\eta\|_{\infty}^2\sum_{\alpha=1}^N\|\xi_{\sigma}(v_m^{\alpha}-v^{\alpha})\|_2^2.
\end{equation*}The strong convergence in~\eqref{eq: usefulconv} implies that $\|\xi_{\sigma}(v_m^{\alpha}-v^{\alpha})\|_2^2\rightarrow 0$ as $m\rightarrow\infty$ for each $\alpha\in\{1,\dots, N\}$. This result and~\eqref{eq: supconv} applied to $\mathcal{K}\equiv\overline{B_r\setminus B_s}$ imply that
\begin{equation}\label{eq: II}
\limsup_{m\rightarrow\infty} II\leq \frac{K^2}{2}\mu(\overline{B_r\setminus B_s}).
\end{equation}
Finally, for III, by (H2), the definition of $\xi_{\sigma}$ and~\eqref{eq: piii}, it follows that
\begin{equation}\label{eq: III}
\limsup_{m\rightarrow\infty}III\leq \limsup_{m\rightarrow\infty} \varepsilon_m K_0\sigma=0.
\end{equation}Next, we take the limit superior in~\eqref{eq: troublesome1} as $m\rightarrow\infty$ to arrive at the following estimate for all $\sigma\in (0, \infty)$ and for all $s\in (0, r)$ using the bounds from~\eqref{eq: I} to~\eqref{eq: III}:
\begin{align}\label{eq: F12}
\limsup_{m\rightarrow\infty}\int_{B_s}A^{\alpha\beta}_{ij, m}(y)D_jv_m^{\beta}(y)D_iv_m^{\alpha}(y)\chi_{\Theta_{\sigma, m}^{\alpha}}(y)\ \mathrm{d}y&\leq\int_{B_r}b^{\alpha\beta}_{ij}D_jv^{\beta}D_iv^{\alpha}\ \mathrm{d}y\nonumber\\
&+\frac{K^2}{2}\mu(\overline{B_r\setminus B_s}).
\end{align}Suppose that for all $s\in (0, r)$ we can establish that
\begin{align}\label{eq: inbetween}
&\lim_{\sigma\rightarrow\infty}\limsup_{m\rightarrow\infty}\int_{B_s}A^{\alpha\beta}_{ij, m}(y)D_jv_m^{\beta}(y)D_iv_m^{\alpha}(y)\chi_{\Theta_{\sigma, m}^{\alpha}}(y)\ \mathrm{d}y\nonumber\\
&=\limsup_{m\rightarrow\infty}\int_{B_s}A^{\alpha\beta}_{ij, m}(y)D_jv_m^{\beta}(y)D_iv_m^{\alpha}(y)\ \mathrm{d}y.
\end{align}Upon sending $\sigma\rightarrow\infty$ in~\eqref{eq: F12}, we would then deduce from~\eqref{eq: inbetween} that
\begin{equation}\label{eq: F1}
\limsup_{m\rightarrow\infty}\int_{B_s}A^{\alpha\beta}_{ij, m}(y)D_jv_m^{\beta}(y)D_iv_m^{\alpha}(y)\ \mathrm{d}y\leq\int_{B_r}b^{\alpha\beta}_{ij}D_jv^{\beta}D_iv^{\alpha}\ \mathrm{d}y+\frac{K^2}{2}\mu(\overline{B_r\setminus B_s})
\end{equation}for each $s<r$. As $s\nearrow r$, it follows that 
\begin{equation}\label{eq: muzero}
\mu(\overline{B_r\setminus B_s})\rightarrow\mu(\partial B_r)=0
\end{equation}by~\eqref{eq: mu}.

Now we consider the function $F: (0, r]\rightarrow (0, \infty)$ defined as
\begin{equation*}
F(s)\equiv\limsup_{m\rightarrow\infty}\int_{B_s}A^{\alpha\beta}_{ij, m}(y)D_jv_m^{\beta}(y)D_iv_m^{\alpha}(y)\ \mathrm{d}y
\end{equation*}and the sequence
\begin{equation*}
I_m\equiv\int_{B_r}A^{\alpha, \beta}_{ij, m}(y)D_jv_m^{\beta}(y)D_iv_m^{\alpha}(y)\ \mathrm{d}y
\end{equation*}for each $m\in\mathbb{N}$. Both are bounded below by $0$ via (H3) and bounded above by $K$ following (H2) and~\eqref{eq: piii}. Moreover, the non-negativity of $F$ in turn implies that it is monotone increasing. As the sequence $\{I_m\}_{m=1}^{\infty}$ is bounded, by passing to a subsequence and relabelling, if necessary, we have that $I_m$ converges to some $I_0\in[0, \infty)$. A consequence of the convergence is that $\limsup_{m\rightarrow\infty}I_m=\lim_{m\rightarrow\infty} I_m$, and a further consequence is that $F$ is defined at $r$ as $I_0$. It remains to show that
\begin{equation*}
\lim_{s\nearrow r} F(s)=F(r).
\end{equation*}Given $s<r$ and $m\in\mathbb{N}$, we define
\begin{equation}\label{eq: jm}
J_m^s\equiv\int_{B_s}A^{\alpha, \beta}_{ij, m}(y)D_jv_m^{\beta}(y)D_iv_m^{\alpha}(y)\ \mathrm{d}y.
\end{equation}As $B_s\subset B_r$ and the integrand in $J_m^s$ is non-negative by (H3), then $J_m^s\leq I_m$ for each $s<r$ and $m\in\mathbb{N}$. Therefore, $J_m^s$ is uniformly bounded above by $K$ also. Consequently, by passing to a subsequence and relabelling, if necessary, we have that $J_m^s$ converges to some $J_0^s$ for each $s\in (0, r)$. Therefore, $\limsup_{m\rightarrow\infty}J_m^s=\lim_{m\rightarrow\infty}J_m^s$ for each $s\in (0, r)$. In summary,
\begin{equation}\label{eq: summary1}
\left\{\begin{array}{l l}
I_0=\lim_{m\rightarrow\infty}I_m=\limsup_{m\rightarrow\infty} I_m& \\
J_0^s=\lim_{m\rightarrow\infty}J_m^s=\limsup_{m\rightarrow\infty}J_m^s & \quad\text{for each }s<r.
\end{array}\right.
\end{equation}Now by (H2) and~\eqref{eq: meas}, we calculate that
\begin{equation}\label{eq: summary2}
\int_{B_r\setminus B_s}A^{\alpha\beta}_{ij, m}(y)D_jv_m^{\beta}(y)D_iv_m^{\alpha}(y)\ \mathrm{d}y\leq K\mu_m(\overline{B_r\setminus B_s})
\end{equation}for each $s<r$. Via~\eqref{eq: supconv} with $\mathcal{K}=\overline{B_r\setminus B_s}$ we come to the conclusion that
\begin{equation}\label{eq: summary3}
\limsup_{m\rightarrow\infty}\mu_m(\overline{B_r\setminus B_s})\leq\mu(\overline{B_r\setminus B_s})\quad (s<r).
\end{equation}From the definition of $F$,~\eqref{eq: summary1},~\eqref{eq: summary2} and~\eqref{eq: summary3}, we arrive at the inequality
\begin{align*}
0\leq F(r)-F(s)&=\limsup_{m\rightarrow\infty}I_m-\limsup_{m\rightarrow\infty}J_m^s\\
&=\lim_{m\rightarrow\infty}(I_m-J_m^s)\\
&=\lim_{m\rightarrow\infty}\int_{B_r\setminus B_s}A^{\alpha\beta}_{ij, m}(y)D_jv_m^{\beta}(y)D_iv_m^{\alpha}(y)\ \mathrm{d}y\\
&\leq\lim_{m\rightarrow\infty}K\mu_m(\overline{B_r\setminus B_s})\\
&\leq\limsup_{m\rightarrow\infty}K\mu_m(\overline{B_r\setminus B_s})\\
&\leq K\mu(\overline{B_r\setminus B_s})
\end{align*}for each $s<r$, and consequently, by~\eqref{eq: muzero} we have
\begin{equation}\label{eq: s2r}
F(r)=\lim_{s\nearrow r}F(s).
\end{equation}Finally, we take $s\nearrow r$ in~\eqref{eq: F1} to conclude by~\eqref{eq: s2r} and~\eqref{eq: muzero} that
\begin{equation*}
\limsup_{m\rightarrow\infty}\int_{B_r}A^{\alpha\beta}_{ij,m}(y)D_jv_m^{\beta}(y)D_iv_m^{\alpha}(y)\ \mathrm{d}y\leq \int_{B_r}b^{\alpha\beta}_{ij}D_jv^{\beta}D_iv^{\alpha}\ \mathrm{d}y.
\end{equation*}
 \end{proof}

It remains to prove~\eqref{eq: inbetween}. For each $(\sigma, s)\in(0, \infty)\times(0, r)$, let
\begin{equation*}
H(\sigma, s)\equiv\limsup_{m\rightarrow\infty}\int_{B_s}A^{\alpha\beta}_{ij, m}(y)D_jv_m^{\beta}(y)D_iv_m^{\alpha}(y)\chi_{\Theta_{\sigma, m}^{\alpha}}(y)\ \mathrm{d}y.
\end{equation*}Furthermore, given $(m, \sigma, s)\in \mathbb{N}\times(0, \infty)\times(0, r)$ let  
\begin{equation*}
\mathfrak{H}_m(\sigma, s)\equiv\sup_{k\geq m}\int_{B_s}A^{\alpha\beta}_{ij, k}(y)D_jv_k^{\beta}(y)D_iv_k^{\alpha}(y)\chi_{\Theta_{\sigma, k}^{\alpha}}(y)\ \mathrm{d}y.
\end{equation*}Then 
\begin{equation}\label{eq: hs}
H(\sigma, s)\equiv\lim_{m\rightarrow\infty}\mathfrak{H}_m(\sigma, s)\quad((\sigma, s)\in(0, \infty)\times(0, r)).
\end{equation}In a similar fashion, for each $s\in(0, r)$ we define
\begin{equation*}
G(s)\equiv\limsup_{m\rightarrow\infty}\int_{B_s}A^{\alpha\beta}_{ij, m}(y)D_jv_m^{\beta}(y)D_iv_m^{\alpha}(y)\ \mathrm{d}y,
\end{equation*}and for each $(m, s)\in \mathbb{N}\times (0, r)$, let
\begin{equation}\label{eq: gsup}
\mathfrak{G}_m(s)\equiv\sup_{k\geq m}\int_{B_s}A^{\alpha\beta}_{ij, k}(y)D_jv_k^{\beta}(y)D_iv_k^{\alpha}(y)\ \mathrm{d}y.
\end{equation}Consequently,
\begin{equation}\label{eq: gs}
G(s)\equiv\lim_{m\rightarrow\infty}\mathfrak{G}_m(s)\quad(s\in (0, r)).
\end{equation}Proving~\eqref{eq: inbetween} is equivalent to proving the following lemma.

\begin{lem}\label{inbtwn}
For each $s\in (0, r)$,
\begin{equation}\label{eq: inbtwn}
\lim_{\sigma\rightarrow\infty}H(\sigma, s)=G(s).
\end{equation}
\end{lem}

\begin{proof}
Let $s\in (0, r)$ be given. We will now prove that~\eqref{eq: inbtwn} holds.

Since~\eqref{eq: usefulconv} implies that $|v_k^{\alpha}-v^{\alpha}|$ is uniformly bounded in $L^1(B)$ for each $\alpha\in\{1, \dots, N\}$, then in turn 
\begin{equation*}
|B\setminus\{y\in B\ |\ |v_k^{\alpha}(y)-v^{\alpha}(y)|<\infty\}| =0.
\end{equation*}Fixing $m\in\mathbb{N}$, we have 
\begin{equation}\label{eq: dct1}
A^{\alpha\beta}_{ij, k}(y)D_jv_k^{\beta}(y)D_iv_k^{\alpha}(y)\chi_{\Theta_{\sigma, k}^{\alpha}}(y)\rightarrow A^{\alpha\beta}_{ij, k}(y)D_jv_k^{\beta}(y)D_iv_k^{\alpha}(y)\quad\text{a.e. }y\in B_s,
\end{equation}as $\sigma\rightarrow\infty$ for every $k\geq m$. Next we observe that, by (H3) and the inequality $0\leq\chi_{\Theta_{\sigma, k}^{\alpha}}\leq 1$, for all $k\geq m$
\begin{equation}\label{eq: dct2}
0\leq A^{\alpha\beta}_{ij, k}(y)D_jv_k^{\beta}(y)D_iv_k^{\alpha}(y)\chi_{\Theta_{\sigma, k}^{\alpha}}(y)\leq A^{\alpha\beta}_{ij, k}(y)D_jv_k^{\beta}(y)D_iv_k^{\alpha}(y)\quad\text{a.e. }y\in B_s
\end{equation}and for all $\sigma>0$. Furthermore, by the non-negative sequence $\{J_m^s\}_{m=1}^{\infty}$ defined in~\eqref{eq: jm}, we have the bound:
\begin{equation}\label{eq: dct3}
\int_{B_s}A^{\alpha\beta}_{ij, k}(y)D_jv_k^{\beta}(y)D_iv_k^{\alpha}(y)\ \mathrm{d}y\leq K,
\end{equation} and therefore,~\eqref{eq: dct1},~\eqref{eq: dct2} and~\eqref{eq: dct3} allow us to deduce, via Lebesgue's dominated convergence theorem, that 
\begin{equation}\label{eq: km}
\lim_{\sigma\rightarrow\infty}\int_{B_s}A^{\alpha\beta}_{ij, k}(y)D_jv_k^{\beta}(y)D_iv_k^{\alpha}(y)\chi_{\Theta_{\sigma, k}^{\alpha}}(y)\ \mathrm{d}y=\int_{B_s} A^{\alpha\beta}_{ij, k}(y)D_jv_k^{\beta}(y)D_iv_k^{\alpha}(y)\ \mathrm{d}y
\end{equation}for each $k\geq m$.

Next we claim that 
\begin{equation}\label{eq: interim}
\lim_{\sigma\rightarrow\infty}\mathfrak{H}_m(\sigma, s)=\mathfrak{G}_m(s).
\end{equation}Suppose that~\eqref{eq: interim} is not true. We note that $\mathfrak{H}_m(\sigma, s)\geq 0$ for each $(m, \sigma, s)\in \mathbb{N}\times (0, \infty)\times(0, r)$. Hence, there exists an $\varepsilon>0$ such that for all $\mathfrak{M}>0$ there is a $\sigma>\mathfrak{M}$ with
\begin{equation}\label{eq: lb1}
\mathfrak{G}_m(s)-\mathfrak{H}_m(\sigma, s)\geq\varepsilon.
\end{equation}For each $k\geq m$,  
\begin{equation}\label{eq: lb2}
\mathfrak{H}_m(\sigma, s)\geq \int_{B_s}A^{\alpha\beta}_{ij, k}(y)D_jv_k^{\beta}(y)D_iv_k^{\alpha}(y)\chi_{\Theta_{\sigma, k}^{\alpha}}(y)\ \mathrm{d}y.
\end{equation}By substituting~\eqref{eq: lb2} into~\eqref{eq: lb1}, we arrive at:
\begin{equation}\label{eq: lb3}
\varepsilon\leq\mathfrak{G}_m(s)-\int_{B_s}A^{\alpha\beta}_{ij, k}(y)D_jv_k^{\beta}(y)D_iv_k^{\alpha}(y)\chi_{\Theta_{\sigma, k}^{\alpha}}(y)\ \mathrm{d}y\quad(k\geq m).
\end{equation}As $\sigma\rightarrow\infty$, 
\begin{equation}\label{eq: lb4}
\varepsilon\leq\mathfrak{G}_m(s)-\int_{B_s}A^{\alpha\beta}_{ij, k}(y)D_jv_k^{\beta}(y)D_iv_k^{\alpha}(y)\ \mathrm{d}y\quad(k\geq m)
\end{equation}by~\eqref{eq: km}.

On the other hand, the definition of $\mathfrak{G}_m$ in~\eqref{eq: gsup} implies that there is some $k_{\ast}\geq m$ such that
\begin{equation}\label{eq: gsup2}
\mathfrak{G}_m(s)-\int_{B_s}A^{\alpha\beta}_{ij, k_{\ast}}(y)D_jv_{k_{\ast}}^{\beta}(y)D_iv_{k_{\ast}}^{\alpha}(y)\ \mathrm{d}y<\varepsilon.
\end{equation}As~\eqref{eq: lb4} is true for all $k\geq m$, it is true for $k_{\ast}$, but this contradicts~\eqref{eq: gsup2}. Therefore, we conclude that~\eqref{eq: interim} is true.

Now given $s\in (0, r),\ \sigma>0$ and $\varepsilon>0$, we know by definitions of $H$ and $G$ that there exists an $m_0=m_0(s, \sigma, \varepsilon)\in\mathbb{N}$ such that for each $m>m_0$ 
\begin{equation*}
0\leq G(s)-H(\sigma, s)<\varepsilon+\mathfrak{G}_m(s)-\mathfrak{H}_m(\sigma, s).
\end{equation*}Via~\eqref{eq: interim} 
\begin{equation*}
0\leq\lim_{\sigma\rightarrow\infty}(G(s)- H(\sigma, s))=G(s)-\lim_{\sigma\rightarrow\infty}H(\sigma, s)<\varepsilon+\mathfrak{G}_m(s)-\lim_{\sigma\rightarrow\infty}\mathfrak{H}_m(\sigma, s)=\varepsilon.
\end{equation*}As $\varepsilon>0$ is arbitrary, we conclude that
\begin{equation*}
G(s)=\lim_{\sigma\rightarrow\infty}H(\sigma, s)\quad(s<r).
\end{equation*}
\end{proof}

\begin{lem}\label{wktostrong}
Locally, the sequence of rescaled gradients $\{Dv_m\}_{m=1}^{\infty}\subset L^2(B(0, 1);\mathbb{R}^{2N})$ converge strongly to $Dv$ in the $L^2$-norm, that is,
\begin{equation*}
Dv_m\rightarrow Dv\quad\text{in }L_{\text{loc}}^2(B(0,1);\mathbb{R}^{2N})
\end{equation*}as $m\rightarrow\infty$.
\end{lem}
\begin{proof}Let $r\in (0, 1)$ such that~\eqref{eq: limsupeq} holds. By (H3),

\begin{align}\label{eq: 234a}
&\lambda\int_{B_r}|Dv_m-Dv|^2\ \mathrm{d}y\\
&\leq\int_{B_r}A^{\alpha\beta}_{ij,m}(y)(D_jv_m^{\beta}(y)-D_jv^{\beta}(y))(D_iv_m^{\alpha}(y)-D_iv^{\alpha}(y))\ \mathrm{d}y\nonumber\\
&=\int_{B_r}A^{\alpha\beta}_{ij,m}(y)D_jv_m^{\beta}(y)D_iv_m^{\alpha}(y)\ \mathrm{d}y+\int_{B_r}A^{\alpha\beta}_{ij,m}(y)D_jv^{\beta}(y)D_iv^{\alpha}(y)\ \mathrm{d}y\nonumber\\
&\ -\int_{B_r}A^{\alpha\beta}_{ij,m}(y)D_jv_m^{\beta}(y)D_iv^{\alpha}(y)\ \mathrm{d}y-\int_{B_r}A^{\alpha\beta}_{ij,m}(y)D_jv^{\beta}(y)D_iv_m^{\alpha}(y)\ \mathrm{d}y\nonumber\\
&\equiv I+II+III+IV.
\end{align} By~\eqref{eq: limsupeq}, we recall that
\begin{equation}\label{eq: 1}
\limsup_{m\rightarrow\infty} I\leq \int_{B_r}b^{\alpha\beta}_{ij}D_jv^{\beta}D_iv^{\alpha}\ \mathrm{d}y.
\end{equation}Using (H2),~\eqref{eq: constconv2} and the fact that $Dv\in L^2(B;\mathbb{R}^{2N})$ we deduce via an application of Lebesgue's dominated convergence theorem that
\begin{equation}\label{eq: 2}
\limsup_{m\rightarrow\infty}II=\lim_{m\rightarrow\infty}II=\int_{B_r}b^{\alpha\beta}_{ij}D_jv^{\beta}D_iv^{\alpha}\ \mathrm{d}y.
\end{equation}We use~\eqref{eq: constconv} with $\phi_i^{\alpha}=D_iv^{\alpha}$ to deduce for III that
\begin{equation}\label{eq: 3}
\limsup_{m\rightarrow\infty}III=\lim_{m\rightarrow\infty}III=-\int_{B_r}b^{\alpha\beta}_{ij}D_jv^{\beta}D_iv^{\alpha}\ \mathrm{d}y,
\end{equation}and similarly, for IV we have that
\begin{equation}\label{eq: 4}
\limsup_{m\rightarrow\infty}IV=\lim_{m\rightarrow\infty}IV=-\int_{B_r}b^{\alpha\beta}_{ij}D_jv^{\beta}D_iv^{\alpha}\ \mathrm{d}y.
\end{equation}

Now we let $m\rightarrow\infty$ in~\eqref{eq: 234a} and use~\eqref{eq: 1} to~\eqref{eq: 4} to conclude that
\begin{equation*}
\lim_{m\rightarrow\infty}\int_{B_r}|Dv_m-Dv|^2\ \mathrm{d}y=0.
\end{equation*}That is, $Dv_m\rightarrow Dv$ in $L^2(B_r;\mathbb{R}^{2N})$ as $m\rightarrow\infty$. Hence, we can pass to a subsequence such that $Dv_m\rightarrow Dv$ a.e. in $B_r$ as $m\rightarrow\infty$. This is true for a.e. $r\in (0, 1)$. Given any ball $B(y_0, \rho)\subset\subset B$, there exists an $r\in (0, 1)$ such that $B(y_0, \rho)\subset\subset B_r$ and $Dv_m\rightarrow Dv$ strongly in $L^2(B_r;\mathbb{R}^{2N})$. Therefore,
\begin{equation*}
Dv_m\rightarrow Dv\quad\text{in }L^2_{\text{loc}}(B;\mathbb{R}^{2N})
\end{equation*}as $m\rightarrow\infty$.
\end{proof}
\begin{rmk}
Some of the ideas in Lemma~\ref{superior} and Lemma~\ref{wktostrong} are from Hamburger's paper~\cite[pp. 25--30]{Hamburger1998}
\end{rmk}
\section{The iteration step}

With the energy-decay estimate (Proposition~\ref{blowup}) in hand, we can now prove Theorem~\ref{regularity}. The iteration step is well known, see, for instance,~\cite[p. 185]{Chen98}. Nonetheless we present it here for the sake of completion.
\begin{proof}
Given $\gamma\in (0, 1)$, choose $\tau=(2c_0)^{2\gamma-2}$ such that:
\begin{equation}\label{eq: 238a}
2c_0\tau^{2}=\tau^{2\gamma},
\end{equation}and note that $\tau\in (0, 1)$.

Now it follows from Proposition~\ref{blowup} that there exists an $\varepsilon_0>0$ and $r_0>0$ such that whenever we have
\begin{equation}\label{eq: 238b}
\Phi(x_0, r)< \varepsilon_0^2
\end{equation}for some $x_0\in\Omega$ and some $r\in (0, \text{min}\{r_0, \text{dist }(x_0, \partial\Omega)\})$, then we have 
\begin{equation}\label{eq: 238c}
\Phi(x_0, \tau r)\leq 2c_0\tau^2\Phi(x_0, r).
\end{equation}As a consequence of~\eqref{eq: 238a},~\eqref{eq: 238b} and the fact that $\tau<1$, we arrive at the following estimate for the right-hand side of~\eqref{eq: 238c}:
\begin{equation*}
2c_0\tau^2\Phi(x_0, r)= \tau^{2\gamma}\Phi(x_0, r)<\varepsilon_0^2,
\end{equation*}and therefore, we can write~\eqref{eq: 238c} as:
\begin{equation*}
\Phi(x_0, \tau r)<\varepsilon_0^2.
\end{equation*}Thus, we can apply Proposition~\ref{blowup} again with $\tau r$ instead of $r$ in~\eqref{eq: 238b} and~\eqref{eq: 238c} to discover that
\begin{equation}\label{eq: 239}
\Phi(x_0, \tau^2 r)\leq2c_0\tau^2\Phi(x_0, \tau r).
\end{equation}We use~\eqref{eq: 238c},~\eqref{eq: 238a},~\eqref{eq: 238b} and the fact that $\tau<1$ in~\eqref{eq: 239}, to see that
\begin{equation}
2c_0\tau^2\Phi(x_0, \tau r)\leq2^2c_0^2\tau^4\Phi(x_0, r)=\tau^{4\gamma}\Phi(x_0, r)<\varepsilon_0^2
\end{equation}Consequently,~\eqref{eq: 239} can be estimated as
\begin{equation}
\Phi(x_0, \tau^2 r)<\varepsilon_0^2,
\end{equation}which allows us to use Proposition~\ref{blowup} once more. After $k$ iterations, we have
\begin{equation}\label{eq: iterlemma}
\Phi(x_0, \tau^kr)< 2^kc_0^k\tau^{2k}\Phi(x_0, r)<\tau^{2\gamma k}\varepsilon_0^2.
\end{equation}

For any $\rho\in (0, r)$, we let $k\in \mathbb{N}_0$ such that
\begin{equation}\label{eq: rhointerval}
\tau^{k+1}r\leq\rho<\tau^kr.
\end{equation}From~\eqref{eq: rhointerval} we observe that
\begin{equation}\label{eq: taumanip}
\tau^{(k+1)2\gamma}\leq \left(\frac{\rho}{r}\right)^{2\gamma}.
\end{equation}Then~\eqref{eq: iterlemma},~\eqref{eq: rhointerval} and~\eqref{eq: taumanip} imply that
\begin{align}\label{eq: gammarho}
\Phi(x_0, \rho)&=\int_{B(x_0, \rho)}|Du|^2\ \mathrm{d}x\nonumber\\
&\leq \int_{B(x_0, \tau^kr)}|Du|^2 \mathrm{d}x\nonumber\\
&=\Phi(x_0, \tau^kr)\nonumber\\
&\leq \tau^{2\gamma k}\Phi(x_0, r)\nonumber\\
&\leq \tau^{-2\gamma}\left(\frac{\rho}{r}\right)^{2\gamma}\Phi(x_0, r).
\end{align}

Therefore, if for some $r\in (0,\min\{r_0, \text{dist }(x_0, \partial\Omega)\})$ we have $\Phi(x_0, r)<\varepsilon_0^2$, then for all $\rho<r$ we have
\begin{equation*}
\Phi(x_0, \rho)\leq C\left(\frac{\rho}{r}\right)^{2\gamma}\Phi(x_0, r),
\end{equation*}where $C=C(N, K, \lambda, \gamma)$.

Note that $\Phi(x, r)$ is continuous in $x$ for each fixed $r>0$. Therefore, if $\Phi(x_0, r)<\varepsilon_0^2$, then there is some ball $B(x_0, \sigma)$ such that 
\begin{equation*}
\Phi(x, r)<\varepsilon_0^2\quad\text{for }x\in B(x_0, \sigma).
\end{equation*}Hence, for all $0<\rho<r$ we have
\begin{equation*}
\Phi(x, \rho)\leq C\left(\frac{\rho}{r}\right)^{2\gamma}\Phi(x, r),
\end{equation*}by~\eqref{eq: gammarho}, which implies that
\begin{equation*}
\int_{B(x, \rho)}|u-u_{x, \rho}|^2\ \mathrm{d}y\leq \kappa\rho^{2+2\gamma},
\end{equation*}for some $\kappa=\kappa(N, K, \lambda, \gamma, r)$. Therefore,
\begin{equation*}
u\in \mathcal{L}^{2, 2+2\gamma}_{\text{loc}}\big(B(x_0, r);\mathbb{R}^N\big)\cong C^{0, \gamma}_{\text{loc}}\big(B(x_0, r);\mathbb{R}^N\big).
\end{equation*} Now we define:
\begin{equation*}
\Omega_0\equiv\left\{x_0\in\Omega\ |\ \Phi(x_0, r)<\varepsilon_0^2 \text{ for some } r<r_0\right\}.
\end{equation*}Then
\begin{enumerate}
\item $\Omega_0\subset\Omega$ and $\Omega_0$ is open,
\item $u\in C^{0, \gamma}_{\text{loc}}(\Omega_0;\mathbb{R}^N)$ and
\item $|\Omega\setminus\Omega_0|=0$ because we have via Poincar\'e's inequality~\eqref{eq: P} that
\begin{equation*}
\Omega\setminus\Omega_{0}\subseteq \left\{ x\in\Omega\ \big|\ \liminf_{r\rightarrow 0}\int_{B(x_0, r)} |Du|^2\ \mathrm{d}x>\varepsilon_0^2\right\}=\emptyset.
\end{equation*}The singular set is empty, and therefore, $u$ is locally H\"older continuous everywhere in $\Omega$.
\end{enumerate}
\end{proof}
\section{Discussion}
Remark~\ref{hindrance} hints at the problems one will encounter in generalising Theorem~{regularity} to higher dimensions, but does not talk about the possibility of extending Theorem~\ref{regularity} to more general systems in two dimensions. What is the role of the variational structure in the proof of Theorem~\ref{regularity}? In particular, is there a fundamental problem with the proof technique presented here if one were to consider systems as in~\cite{GQ1978}:
\begin{equation}\label{eq: generalg}
-\text{div }(A^{\alpha\beta}_{ij}(x, u)D_ju^{\beta})=g^{\alpha}(x, u, Du)\quad(\alpha=1, \dots, N)?
\end{equation} In two dimensions, the answer depends crucially on the exact natural growth conditions that we impose. If we impose the strong conditions
\begin{equation}\label{eq: thaakat}
|A(x, z)\zeta|\leq K|\zeta| \quad\text{and}\quad |g(x, z, \zeta)|\leq K_0|\zeta|^2
\end{equation}for all $(x, z, \zeta)\in \overline{\Omega}\times\mathbb{R}^N\times\mathbb{R}^{2N}$ and some $K, K_0>0$, then the arguments carry forward with no trouble. In this respect, our result is true for systems of the type~\eqref{eq: generalg} that are not necessarily variational. 

However, we point out that this follows from the \emph{imposition} of~\eqref{eq: thaakat}. Returning to our variational setting, because the coefficients are smooth and bounded in $\overline{\Omega}\times\mathbb{R}^N\times\mathbb{R}^{2N}$, the principal part always satisfies the inequality
\begin{equation}\label{eq: hamesha}
|A(x, z)\zeta|\leq K|\zeta|
\end{equation}for all $(x, z, \zeta)\in \overline{\Omega}\times\mathbb{R}^N\times\mathbb{R}^{2N}$ and some $K>0$, and via Cauchy-Schwarz, the inhomogeneity in~\eqref{eq: maineq} always satisfies the inequality
\begin{equation*}
\left|\frac{1}{2}\frac{\partial A^{\gamma\beta}_{ij}}{\partial z^{\alpha}}(x, z)\zeta_j^{\beta}\zeta_i^{\gamma}\right| \leq K_0|\zeta|^2
\end{equation*}for all $(x, z, \zeta)\in \overline{\Omega}\times\mathbb{R}^N\times\mathbb{R}^{2N}$ and some $K_0>0$. Thus, the variational structure allows us to \emph{deduce} rather than \emph{impose}~\eqref{eq: thaakat} at least for the inhomogeneity.

Typically, in the literature one assumes the weaker natural growth conditions:
\begin{equation}\label{eq: kamzur}
|A(x, z)\zeta|\leq K(|\zeta|+1) \quad\text{and}\quad |g(x, z, \zeta)|\leq K_0(|\zeta|^2+1)
\end{equation}for all $(x, z, \zeta)\in \overline{\Omega}\times\mathbb{R}^N\times\mathbb{R}^{2N}$ and some $K, K_0>0$, see, for example,~\cite[p. 4]{GQ1978}. Clearly, our variational system satisfies~\eqref{eq: kamzur}, but that is inconsequential as we employ~\eqref{eq: thaakat} in the proof in any case. However, if we pass to the system~\eqref{eq: generalg} under the natural growth conditions~\eqref{eq: kamzur}, then we do run into problems. The proof follows through without any issue until we arrive at Lemma~\ref{homogsys}. The reason is simply because we do not need to work with the inhomogeneity until Lemma~\ref{homogsys}. Let us perform the steps in Lemma~\ref{homogsys}. 

Given any $B(x_m, r_m)$ and any test function $\tilde{\varphi}\in C_c^{\infty}(B(x_m, r_m);\mathbb{R}^N)$, we have that 
\begin{align*}
I&\equiv\int_{B(x_m, r_m)}A^{\alpha\beta}_{ij}(x, u(x))D_ju^{\beta}(x)D_i\tilde{\varphi}^{\alpha}(x)\ \mathrm{d}x\\
&=\int_{B(x_m, r_m)}g^{\alpha}(x, u(x), Du(x)) \tilde{\varphi}^{\alpha}(x)\ \mathrm{d}x\equiv II.
\end{align*}By rescaling from $B(x_m, r_m)$ to $B$, we find that
\begin{equation*}
I=\varepsilon_m\int_{B(0, 1)}A^{\alpha \beta}_{ij,m}(y)D_jv_m^{\beta}(y)D_i\tilde{\varphi}^{\alpha}(x_m+r_my)\ \mathrm{d}y
\end{equation*}and
\begin{equation*}
II=r_m^{2}\int_{B(0, 1)}g^{\alpha}(x_m+r_my, \varepsilon_mv_m(y)+u_{x_m, r_m}, \varepsilon_mr_m^{-1}Dv_m(y))\tilde{\varphi}^{\alpha}(x_m+r_my)\ \mathrm{d}y. 
\end{equation*}Writing $\varphi_m(y)$ for $\tilde{\varphi}(x_m +r_my)$ and equating $I$ to $II$, we see that
\begin{align}\label{eq: yeqg1}
&\int_{B(0, 1)}A^{\alpha \beta}_{ij,m}(y)D_jv_m^{\beta}(y)D_i\varphi_m^{\alpha}(y)\ \mathrm{d}y\nonumber\\
&\qquad=\varepsilon_m^{-1}r_m^{2}\int_{B(0,1)}g^{\alpha}(x_m+r_my, \varepsilon_mv_m(y)+u_{x_m, r_m}, \varepsilon_mr_m^{-1}Dv_m(y))\varphi^{\alpha}_m(y)\ \mathrm{d}y.
\end{align}

Given any $\varphi\in C_c^{\infty}(B(0, 1);\mathbb{R}^N)$, we arrive at the following estimate for the left-hand side of~\eqref{eq: yeqg1} by~\eqref{eq: kamzur} and~\eqref{eq: piii}:
\begin{align}\label{eq: bdg1}
\left|\int_{B(0, 1)}A^{\alpha \beta}_{ij,m}(y)D_jv_m^{\beta}(y)D_i\varphi^{\alpha}(y)\ \mathrm{d}y\right|&\leq (\varepsilon_mK_0\|Dv_m\|_{L^2(B)}^2+\alpha(2)K_0\varepsilon_m^{-1}r_m^2)\|\varphi\|_{\infty}\nonumber\\
&\leq(\varepsilon_mK_0+\alpha(2)K_0\varepsilon_m^{-1}r_m^2)\|\varphi\|_{\infty}.
\end{align}Consequently, as $m\rightarrow\infty$, the right-hand side of~\eqref{eq: bdg1} does not necessarily vanish in contrast to the variational setting. Without being able to show that $v$ weakly solves the linear system
\begin{equation*}
\text{div }(b^{\alpha\beta}_{ij}D_jv^{\beta})=0
\end{equation*}in $B(0, 1)$ we cannot complete the blow-up argument. We need another way to bound the left-hand side of~\eqref{eq: bdg1}. Clearly the most natural way to do this is to bound the inhomogeneities such that they vanish in the limit, and having the strong natural growth conditions~\eqref{eq: thaakat} would achieve this.


\bibliographystyle{spmpsci}      

\bibliography{refs}   

\end{document}